\documentclass[12pt,reqno]{amsart}
\usepackage{a4wide}
\usepackage{xcolor}
\usepackage{url}
\usepackage{amsmath,amsthm,bm}
\usepackage{amssymb}
\usepackage[utf8]{inputenc}
\usepackage[T1]{fontenc}
\def\F{\mathbb{F}}
\def\Z{\mathbb{Z}}
\def\C{\mathbb{C}}
\def\E{\mathbb{E}}
\def\P{\mathbb{P}}
\newcommand{\nor}[1]{\lVert #1 \rVert}

\def\rk{\text{rk}}
\def\Span{\mathrm{span}}
\def\Mat{\mathrm{Mat}}
\def\codim{\mathrm{codim}}
\newcommand{\abs}[1]{\left\lvert #1 \right\rvert}
\newtheorem{trm}{Theorem}
\newtheorem{prop}[trm]{Proposition}
\newtheorem{conj}[trm]{Conjecture}
\newtheorem{lm}[trm]{Lemma}
\newtheorem{cor}[trm]{Corollary}
\theoremstyle{remark}

\newtheorem{rmk-pl}[rmk]{Remarks}

\newtheorem{ex-sg}[ex]{Example}
\theoremstyle{definition}

\title{A bilinear Bogolyubov theorem}
\author[P.-Y. Bienvenu]{Pierre-Yves Bienvenu}
\address{P.-Y. Bienvenu, Institut Camille-Jordan, Université Lyon 1, 43 boulevard du 11 novembre 1918
69622 Villeurbanne cedex, France}
\email{pierre.bienvenu@math.univ-lyon1.fr}
\title{A bilinear Bogolyubov theorem}
\author[T. H. L\^e]{Th\'ai Ho\`ang L\^e}
\address{T. H. L\^e, Department of Mathematics, The University of Mississippi, University, MS 38677, United States}
\email{leth@olemiss.edu} 
\date{\today}\begin{document}
\begin{abstract}
The purpose of this note is to prove the existence of
a remarkable structure in an
iterated sumset derived from a set $P$ in
a Cartesian square $\F_p^n\times\F_p^n$.
More precisely, we perform horizontal and vertical sums and differences on $P$, that is, operations on the second coordinate
when the first one is fixed, or vice versa.
The structure we find is the zero set of a family of bilinear forms
on a Cartesian product of vector subspaces. 
The codimensions of the subspaces and the number of bilinear forms
involved are bounded by a function $c(\delta)$
of the density $\delta=\abs{P}/p^{2n}$ only.
The proof uses various tools of additive combinatorics, such as the (linear) Bogolyubov theorem,
the density increment method, as well as the Balog-Szemer\'edi-Gowers and Freiman-Ruzsa theorems.
\end{abstract}

\maketitle
\section{Introduction}
Let $\F=\F_p$ be a finite field of fixed prime order and $V$ be a vector space of dimension $n$ over $\F$. In this paper, 
the dimension $n$ is the asymptotic parameter and
all the $O(\cdot)$ and $o(\cdot)$ may depend on $p$ but not $n$.
By the \textit{density} of a subset $A \subset V$ we mean 
$\frac{|A|}{|V|}$.
The classical Bogolyubov theorem states the following:
\begin{trm} [Bogolyubov] \label{trm:linBogo} If $A \subset V$ is a set of density $\alpha >0$, then the sumset
\[
A+A-A-A : = \{a_1 + a_2 -a_3-a_4 \mid (a_1,\ldots,a_4) \in A^4 \}
\]
contains a vector subspace of codimension $c(\alpha)$.
\end{trm}The notation $A+A-A-A$ is often abbreviated as $2A-2A$.
Note that Bogolyubov's original argument \cite{bogo} works in $\Z$
instead of
vector spaces, but at least since Green's survey \cite{GreenSurvey}
on finite field models, it has been applied to the vector space setting, where it gives $c(\delta) = O(\delta^{-2})$.  The best bound is due to Sanders \cite[Theorem 11.1]{Sanders}, who showed that $c(\delta) = O(\log^4 \delta^{-1})$. The very short proof of the polynomial bound
as well as a simplified exposition of Sanders' 
breakthrough can be found in the excellent survey \cite{ffsurvey}.
Sanders' result is usually stated and proven for $p=2$ but
holds for any $p$, the implied constant depending on $p$.
We point out that the corresponding statement for $A-A$ is not true both in vector spaces (see \cite[Theorem 9.4]{Green}) and in $\Z$ (a result of K\v{r}\'i\v{z} \cite{kriz}). 

The purpose of this note is to prove a bilinear version of Theorem \ref{trm:linBogo}, that is, for dense subsets of $V \times V$. Before stating our result, we first need some definitions. 
For a subset $P\subset V\times V$, we define vertical and horizontal additive operations
on $P$ as follows:
$$P\stackrel{V}{\pm}P=\{(x,y_1\pm y_2)\mid \left((x,y_1),(x,y_2)\right)\in P^2\}$$ 
and
$$P\stackrel{H}{\pm}P=\{(x_1\pm x_2,y)\mid \left((x_1,y),(x_2,y)\right)\in P^2\}$$ 
where $V$ and $H$ mean vertical and horizontal.
Note that $V$ is also the name of the ambient space, but this
should not create any confusion.
We denote by $\phi_V$
the operation
\begin{equation*}
P\mapsto (P\stackrel{V}{+}P)\stackrel{V}{-}(P\stackrel{V}{+}P)
\end{equation*}
and define the operation $\phi_H$ similarly.
\begin{trm}
\label{mainTrm}
For any $\delta >0$, there exists a constant $c(\delta)>0$
such that the following holds.
Let $P\subset V\times V$ have density $\delta$.
There exist subspaces $W_1\leq V, W_2\leq V$ of codimension
$r_1,r_2$
and a family
$\mathbf{Q}=(Q_1,\ldots,Q_{r_3})$ of bilinear forms on $W_1\times W_2$
such that 
\begin{equation} \label{eq:bilinearSet}
\phi_H \phi_V \phi_H (P) \supset \{(x,y)\in X\times Y\mid Q_1(x,y)=\cdots=
Q_{r_3}(x,y)=0\}
\end{equation}
where
$\max(r_1,r_2, r_3)\leq c(\delta)$.
Moreover $c(\delta)$ can be taken as $O(\exp(\exp(\exp(\log^{O(1)} \delta^{-1}))))$.
\end{trm}

Our proof actually gives $\max(r_1, r_3) = O(\log^{O(1)}\delta^{-1})$. We point out that Gowers and Milićević \cite{GM} independently proved 
a result very similar to Theorem \ref{mainTrm}. However, their method and bounds are different from ours. They proved $\max(r_1,r_2, r_3)
=O ( \exp(\exp(\log^{O(1)}\delta^{-1})))$.

In view of our bounds for $r_1$ and $r_3$ and the fact that the roles of $r_1$ and $r_2$ are symmetric, it is quite reasonable to conjecture the following.
\begin{conj}[polylogarithmic bilinear Bogolyubov] \label{conj}
In Theorem \ref{mainTrm}, one can take $c(\delta)= O (\log^{O(1)}\delta^{-1})$.
\end{conj}
If $P$ is a Cartesian product $A\times B$ for some subsets $A,B\subset V$, then using Theorem \ref{trm:linBogo} once on each coordinate, we obtain a product $A'\times B'$ of subspaces of codimension $O(\log^4\delta^{-1})$.
Also it is easy to see that $c(\delta)\gg \log\delta^{-1}$ by considering a set such as the right-hand side of equation \eqref{eq:bilinearSet}. Conjecture \ref{conj} says that, like in the linear case, this lower bound on $\delta$ should not be too far off the truth. The conjecture remains equally interesting and useful for the application we have in mind if $O(1)$ operations $\phi_V$ or $\phi_H$ are required instead of 3.


\noindent \textbf{A quick application.}
Our application concerns matrices of low rank.
Suppose a two-parameter, bilinearly varying family of matrices is often 
of rank at most $\epsilon$. Then it must be of rank $O(\epsilon)$ on 
a whole bilinear set. 
We now state this application precisely.
\begin{cor}
Suppose that we have a bilinear map
$\psi : \F^n\times\F^n\rightarrow \Mat_m(\F)$.
Suppose that the set
$$
P_\epsilon=\{(f,g)\in\F^n\times\F^n\mid \text{\emph{rank}}(\psi(f,g))\leq\epsilon\}
$$
has density $\delta >0$.
Then the set 
$$
P_{64\epsilon}=\{(f,g)\in\F^n\times\F^n\mid \text{\emph{rank}}(\psi(f,g))\leq 64\epsilon\}
$$
contains a set of the form \eqref{eq:bilinearSet}
where the codimensions and the cardinality of the family of bilinear forms are at most $c(\delta)$.
\end{cor}
The authors exploit this corollary, with the conjectured bound on $c(\delta)$ in Theorem \ref{mainTrm}, in a companion paper \cite{bl}.
\begin{proof}
We apply Theorem \ref{mainTrm} to $P$. Note that the set
$P'$ it produces is included in $P_{64\epsilon}$
by the bilinearity of $\psi$ and the fact that $\text{rank} (A+B)\leq \text{rank } A+\text{rank }B$ for any two matrices $A$ and $B$.
\end{proof}
The organization of the paper is as follows. In Section \ref{sec:prelim} we recall some basic facts and preliminaries. The heavy lifting part of our argument is an iteration scheme, Proposition \ref{codimensionReduction}. In Section \ref{sec:deduce} we show how this Proposition implies Theorem \ref{mainTrm}. Section \ref{sec:proof} is devoted to proving Proposition \ref{codimensionReduction}.
 
\section{Preliminaries} \label{sec:prelim}
The symbol $\E$ will be at some point used in its usual 
probabilistic sense, but it will frequently denote an average,
thus $\E_{x\in X} =\frac{1}{\abs{X}}\sum_{x\in X}$.
Similarly, $\P_{x\in X}(x\in Y)=\frac{\abs{Y}}{\abs{X}}$
for sets $Y\subset X$. 

We now briefly recall some basic facts about
the Fourier transform and convolutions.
Let $V$ be a finite $\F$-vector space; in fact, all spaces considered will be finite in this paper, so we may not always specify this hypothesis. Then we denote by $\widehat{V}$
its \emph{dual}, the set of characters on $V$. 
A character $\chi\in\widehat{V}$ takes values in the $p$-th roots of unity, that is, 
$1,\omega,\ldots,\omega^{p-1}$ where $\omega=\exp(2i\pi x/p)$.
The trivial character is $\chi = 1$.
Let $f : V\rightarrow\C$ be a function.
Then the Fourier transform $\hat{f}$ is
defined on $\widehat{V}$ by
$$
\hat{f}(\chi)=\E_{x\in V}f(x)\chi(x).
$$
In particular, if $A\subset V$ has density $\alpha$ 
and indicator function $1_A$, 
we have $\widehat{1_A}(1)=\alpha$.
Besides, we have $\widehat{1_{-A}}
=\overline{\widehat{1_A}}$.

If $W$ is an affine subspace of $V$ of direction $\overrightarrow{W}$, thus $W=a+\overrightarrow{W}$ for some $a\in V$, and $f : W\rightarrow\C$ is a function,
we define the function $\tilde{f}$
on the vector space $\overrightarrow{W}$ by
$\tilde{f}(v)=f(a+v)$.
We then define the Fourier transform of $f$
relative to $W$ as the Fourier transform of $\tilde{f}$ on
$\overrightarrow{W}$.
We will abuse notation and denote by
$\widehat{W}$  the dual of $\overrightarrow{W}$.
Thus the notion of Fourier transform depends on
the ambient (potentially affine) space
one is considering, but when no ambiguity is possible,
the space considered may not be made explicit.

Besides, if $f,g : V\rightarrow\C$ are two functions,
we define their \emph{convolution} $f* g: V\rightarrow\C $ by
$$
f* g (x)=\E_{y\in V}f(y)g(x-y).
$$
We define the $U^2$ norm by 
$$
\nor{f}_{U^2(V)}^4=\E_{x\in V} \abs{f* f(x)}^2.
$$
A quadruple $(x_1,x_2,x_3,x_4)\in V^4$ satisfying
$x_1+x_2=x_3+x_4$ is called an \emph{additive quadruple}.

Observe that if $f=1_A$ is the indicator function of the subset
$A\subset V$, then 
$$
\nor{1_A}_{U^2(V)}^4=\frac{\abs{\{(x_1,x_2,x_3,x_4)\in A^4\mid x_1+x_2=x_3+x_4\}}}{\abs{V}^3}
$$
and we refer to this quantity as the \emph{density of additive quadruples} in $A$.
Again if $W$ is an affine subspace of $V$ and $f : W\rightarrow\C$ is a function, we
will write $
\nor{f}_{U^2(W)}=\nor{\tilde{f}}_{U^2(\overrightarrow{W})}$.
Remark that the connection with the additive quadruples
of $A\subset W$ is preserved,
because additive quadruples are invariant by translation.
When it is obvious from the context which space one is considering,
one will simply write $\nor{f}_{U^2}$.

We recall without proof a few basic properties of the Fourier transform.
\begin{enumerate}
\item 
Parseval's identity is the statement that
$$
\E_{x\in V}\abs{f(x)}^2=\sum_{\chi\in \widehat{V}}\abs{\hat{f}(\chi)}^2.
$$
In particular, for a  subset
$A\subset V$ of density $\alpha$, we have $\sum_{\chi\in \widehat{V}}\abs{\widehat{1_A}(\chi)}^2=\alpha$.
\item 
The Fourier transform of a convolution is the product of the
Fourier transforms,
that is
$$\widehat{f* g}=\hat{f}\hat{g}.$$
\item \label{p3} Combining the previous two points, we see that
the $U^2$ norm of a function is the $L_4$ norm
of its Fourier transform, that is
$$
\nor{f}_{U^2(V)}=\nor{\hat{f}}_4.
$$
In particular if $f=1_A$ for a subset $A$ of density $\alpha$, Parseval's identity implies that 
\begin{equation} \label{eq:u2}
\alpha^4 \leq \nor{\widehat{1_A}}_4^4 = \alpha^4 + \sum_{\chi\in \widehat{V}, \, \chi \neq 1} \abs{\widehat{1_A}(\chi)}^4 \leq \alpha^4 + \alpha \max_{\chi\in \widehat{V}, \, \chi \neq 1} \abs{\widehat{1_A}(\chi)}^2.
\end{equation}
When a set $A\subset W$ of density $\alpha$ has about as few additive quadruples as it can,
that is, $\alpha^4 \leq \nor{1_A}_{U^2(W)}^4\leq \alpha^4(1+\epsilon)$, we will call it $\epsilon$-\textit{pseudorandom}. In particular, $A$ is 
$\epsilon$-pseudorandom in $W$ if $\max_{\chi\in \widehat{W}, \, \chi \neq 1} \abs{\widehat{1_A}(\chi)} \leq \alpha^{3/2} \epsilon^{1/2}$.

\item The Fourier inversion formula is the statement that
\begin{equation}
\label{inversion}
f=\sum_{\chi\in\widehat{V}}\hat{f}(\chi)\overline{\chi}.
\end{equation}
\end{enumerate}

Our first lemma says that if $A$ is sufficiently pseudorandom in terms of its density then $2A-2A$ is the whole space.

\begin{lm} \label{lem:wholespace}
Let $W$ be an affine subspace of $V$ and $A\subset W$ have
density $\alpha$. If $\|1_A - \alpha \|_{U^2(W)} < \alpha$, or equivalently,
\begin{equation}
\sum_{\chi\neq 1} \abs{\widehat{1_A}(\chi)}^4 < \alpha^4, 
\end{equation}
then $2A-2A = \overrightarrow{W}$. Consequently, if $\max_{\chi \in \widehat{W}, \, \chi \neq 1} \abs{ \widehat{1_A}(\chi)} < \alpha^{3/2}$ then $2A-2A = \overrightarrow{W}$.
\end{lm}
\begin{proof}
For any $x \in \overrightarrow{W}$, by the Fourier inversion formula \eqref{inversion}, we have
\begin{equation*}
1_A* 1_A* 1_{-A}* 1_{-A} (x) = 
\sum_{\chi\in \widehat{W}} \abs{\widehat{1_A}(\chi)}^4\overline{\chi(x)}
\geq \alpha^4-\sum_{\chi\neq 1} \abs{\widehat{1_A}(\chi)}^4 >0.
\end{equation*}
This implies that $x \in 2A - 2A$.
\end{proof}

We also need the following standard fact which relates the lack of pseudorandomness to density increment.

\begin{lm}[{\cite[Lemma 3.4]{GreenSurvey}}] \label{lem:densityincrement}
Let $W$ be an affine subspace of $V$ and $A\subset W$ have
density $\alpha$. Suppose there exists $\chi \in \widehat{W}, \chi\neq 1$ such that $\abs{\widehat{1_A}(\chi)}\geq \beta$. Then there exists an affine subspace $H\leq W$
of codimension 1 such that the density of $A\cap H$ on $H$ is at least $\alpha + \beta/2$. 
\end{lm}

Our next tool is a regularity lemma.  

\begin{lm}
\label{ExtendedRegularity}
Let $W$ be an affine subspace of $V$ and $A\subset W$ have
density $\alpha$. Let $\epsilon >0$.
For any $t$, there exists an affine subspace $H\leq W$ of codimension 
$O(t \epsilon^{-1}\log\alpha^{-1})$ such 
that $\abs{A'}= \alpha'\abs{H}$ (where $A'=A\cap H$) with $\alpha'\geq \alpha$
and for any affine subspace $F$ of codimension at most $t$ of $H$, $\frac{\abs{A\cap F}}{\abs{F}}\leq \alpha(1+\epsilon)$.
Consequently, for any affine subspace $F$ of codimension at most 
$t$ of $H$, we also have 
$\frac{\abs{A\cap F}}{\abs{F}} \geq \alpha (1-p^t \epsilon)$.
\end{lm}
\begin{proof} 
Let us prove the first conclusion. 
If $W$ does the trick already, we do nothing.
If not, there exists a subspace $H$ of codimension at most $t$
such that $\frac{\abs{A\cap H}}{\abs{H}}>\alpha(1+\epsilon)$.
We replace $W$ by $H$, and $A$ by $A\cap H$.
And we iterate.
We duplicate  the density in at most $\epsilon^{-1}$ iterations. And we may duplicate up to $\log\alpha^{-1}$ times before hitting 1.
At every iteration we may lose up to $t$ dimensions.
Whence the first conclusion. The second conclusion follows from summing the upper bound over all cosets of $F$.
\end{proof}

In particular, when $t=1$, the following corollary says that we can always suppose that a set $A \subset W$ is pseudorandom, at the cost of passing to a subset in an affine subspace.

\begin{cor} \label{cor:regularity} Let $W$ be an affine subspace of $V$ and $A\subset W$ have
density $\alpha$. Let $\epsilon >0$.
Then there exists an affine subspace $H\leq W$ of codimension 
$O( (\alpha\epsilon)^{-1/2}\log\alpha^{-1})$ such 
that $A'=A\cap H$ has density $\geq \alpha$ and is $\epsilon$-pseudorandom in $H$.
\end{cor}
\begin{proof}
We use Lemma 6 with $\beta=\alpha^{3/2}\epsilon^{1/2}$,
and Lemma 7 with $t=1$ and $\epsilon'=\alpha^{1/2}\epsilon^{1/2}/2$ to obtain the conclusion.
\end{proof}

Our next tool is a standard lemma resulting from the combination
of the Balog-Szemer\'edi-Gowers and Freiman-Ruzsa theorems.
A useful reference for this lemma is \cite[Lecture 2]{Green}. We reproduce the proof as 
we want to incorporate the quasipolynomial bound
of Sanders \cite[Theorem 11.4]{Sanders}
for the Freiman-Ruzsa theorem.
\begin{lm}
\label{BSGFR}
Let $W\leq V$ be $\F$-vector spaces and $A\subset W$ have
density $\alpha$. 
Let $c>0$ be a constant.
Suppose $\xi : A\rightarrow V$ is such that are at least $c\abs{A}^{3}$ additive quadruples in the graph  
$\Gamma=\{(y,\xi(y))\mid y\in A\}$.
Then there is a subset $S\subset A$ such that
$\xi_{\mid S}$ coincides with an affine-linear map.
Moreover, the density of $S$ in $A$ can be taken
quasipolynomial in $c$, that is $\abs{S}\gg \abs{A}\exp(-\log^{O(1)} c^{-1})$.
\end{lm}
\begin{proof}
First, the Balog-Szemer\'edi-Gowers theorem implies that there exists a set $A'\subset A$
satisfying  $\abs{A'}\geq C\abs{A}$ that induces
a subgraph $\Gamma'\subset\Gamma$ satisfying $\abs{\Gamma'+\Gamma'}\leq C'\abs{\Gamma}$, where both $C$ and $C'$ can be taken polynomial in $c$. 
Using the Freiman-Ruzsa theorem (with Sanders' bounds
from \cite[Theorem 11.4]{Sanders}), 
we get a subgraph $\Gamma''\subset \Gamma'$ corresponding to
a subset $A''\subset A'$
satisfying $\abs{\Gamma''}\geq D\abs{\Gamma'}$ 
and $\abs{\Span(\Gamma'')}\leq E\abs{\Gamma''}$
with $D$ polynomial and $E$ quasipolynomial 
in $c$.
Write $H=\Span(\Gamma'')\leq W\times V$ 
and $\pi : H\rightarrow W$ the canonical projection of $W\times V$ to the first coordinate restricted to $H$.
Then $\pi(H)\supset A''$ by definition.
Because $\abs{H}\leq E\abs{A''}$,
the size of the kernel of $\pi$ is at most $E$.
Then we can partition $H$ into at most $E$ cosets of some subspace $H'$ so that $\pi$ is injective on each of them.
By the pigeonhole principle, there exists such a coset that has a large intersection with $\Gamma''$, that is,
an $x\in W$ such that
$$\abs{(x+H')\cap \Gamma''}\geq \abs{\Gamma''}/E.$$
Let now $\Delta=(x+H')\cap \Gamma''$ and $S$ be the corresponding subset of $A''$. The map $\pi_{\mid x+H'}$
is a bijection onto its image, an affine space $M\leq V$.
Its inverse function is an affine map $\psi : M\rightarrow W$ such that $(s,\psi(s))\in \Gamma''$ for 
all $s\in S$, that is, $\psi(s)=\xi(s)$.
Moreover,
$$\abs{S}=\abs{\Delta}\geq\abs{A''}/E\geq K\abs{A}$$
where $K$ is quasipolynomial in $c$.
\end{proof}
We will also need the following lemma. 
\begin{lm} \label{lem:partition}
Let $A$ be a finite set and $T \subset \{ (a_1, a_2, a_3, a_4) \in A^4 \mid \textup{$a_i$ are pairwise distinct} \}$. Then there is a partition $A = \cup_{i=1}^4 A_i$ such that
$|T \cap A_1 \times A_2 \times A_3 \times A_4 | \geq |T|/256.$
\end{lm}
\begin{proof}
We use the probabilistic method. For a random partition of $A$ where each $y\in A$ is assigned 
a part $A_i$ with $i\in\{1,2,3,4\}$ chosen independently, uniformly with probability 1/4, we have
\begin{equation}
\label{expectation}
\E[\abs{T\cap A_1\times A_2\times A_3\times A_4}]=
\E \sum_{(x_1,\ldots,x_4)\in T}\prod_{i=1}^4 1_{x_i\in A_i}
=\sum_{(x_1,\ldots,x_4)\in T}
\E\prod_{i=1}^4 1_{x_i\in A_i}.
\end{equation}
Let $(x_1,\ldots,x_4)\in T$. In particular the $x_i$
are pairwise distinct.
Then by uniform distribution and independence, 
for any $(m_1,\ldots,m_4)\in[4]^4$,
we have
$$
\P (x_i\in A_{m_i}\text{ for each }i\in [4])
=4^{-4}.
$$
Together with equation \eqref{expectation},
this implies that
$$
\E[\abs{T\cap A_1\times A_2\times A_3\times A_4}]=\abs{T}/256
$$
so there must be a partition with $\abs{T\cap A_1\times A_2\times A_3\times A_4} \geq \abs{T}/256$.
\end{proof}

\section{Deducing the main theorem from an iteration scheme} \label{sec:deduce}
Let $P\subset V\times V$ have density $\delta >0$.
We first apply the linear Bogolyubov theorem (Theorem \ref{trm:linBogo}). Write $P=\cup_{y\in V}B_y\times \{y\}$.
Because $P$ has density $\delta$, the set $A$ of
elements $y\in V$ such that 
$\abs{B_y}\geq \delta \abs{V}/2$ has density at least $\delta /2$.
Using Theorem \ref{trm:linBogo} on each set $B_y$ for $y\in A$, we see that
$\phi_H(P)$
contains a set $P'=\cup_{y\in A}V_y\times \{y\}$
where $V_y$ is a subspace
of codimension $O(\log^{4}1/\delta)$. 

From now on, we will assume that $P=\cup_{y\in A}V_y\times \{y\}$, and we will show that $\phi_H \phi_V (P)$ contains the desired bilinear structure. We achieve this through the following iteration scheme.
Let $V^{*}$ be the linear dual of $V$, that is, the set of linear forms on $V$. For $(x,\xi)\in V\times V^{*}$, 
we denote $x\cdot\xi=\xi(x)$. For a set $U \subset V$, we let $U^\perp=\{\xi\in V^*\mid \forall x\in U,\, x\cdot\xi=0\}$. Also, for a set $T \subset V^*$, we let $T^\perp=\{x \in V \mid \forall \xi \in T,\, x\cdot\xi=0\}$.

\begin{prop}
\label{codimensionReduction}
Let $V$ be an $\F$-vector space, and $W$ be an affine subspace. 
Let $r\leq \dim V$ be an integer and $\alpha>0$. 
Let $\epsilon=p^{-r}/256$.
Then there exists
a constant $c(r,\alpha)$ such that the following holds.
Let $A\subset W$ 
be an $\epsilon$-pseudorandom subset
of density $\alpha$.
Let $P=\cup_{y\in A}V_y\times\{y\} \subset V \times W$ where each $V_y$ is a subspace
of codimension at most $r$.
Suppose there exist $s\leq r$ and affine maps $\xi_1,\ldots,\xi_s$ from
$W$ to $V^*$ and 
spaces $U_y\leq V^*$ for $y\in A$ of dimension at most $r-s$
such that $V_y^\perp=\Span(\xi_j(y))_{j\in [s]} + U_y$.
Then at least one of the following statements holds. 
\begin{enumerate}
\item (Termination) The set $\phi_V(P)$ contains
$$
\{(x,y)\in X_3\times W_2\mid x\cdot \overrightarrow{\xi_1}(y)=\cdots=
x\cdot \overrightarrow{\xi_s}(y)=0\}
$$
where $\overrightarrow{\xi}$ denotes the linear part of
an affine map $\xi$,
$W_2$ is the direction of $W$ and $X_3$ is a subset of density at least $p^{-r}/12$ in $V$.
\item (Reduction of codimension) There exist a set $S\subset A$ of density $c(r,\alpha)$, a subspace $V'\leq V$ of codimension at most $3r+10$ and subspaces $V'_y\leq V'$ of codimension at most $r-1$ (inside $V'$) for each 
$y\in S$ such that $P\supset \bigcup_{y\in S} V'_y\times \{y\}$.
Besides there exist
affine maps $\xi'_1,\ldots,\xi'_{s-1}$ from
$W$ to $V'^*$ and 
spaces $U'_y\leq V'^*$ for $y\in S$ of dimension at most $r-s$
such that $(V'_y)^\perp=\Span(\xi_j(y))_{j\in [s-1]} + U'_y$ (where $(V'_y)^\perp$ is defined inside $V'^*$).
\item (Linearisation) There exist a set $S\subset A$ of density $c(r,\alpha)$ and
an affine map $\xi_{s+1}:W\rightarrow V$ and a space $U'_y< U_y$ of dimension at most $r-s-1$ such that 
for all $y\in S$, we have $V_y^\perp= \Span(\xi_1(y),\ldots,\xi_{s+1}(y))+U'_y$.
\end{enumerate}
Moreover $c(r,\alpha)$ can be taken quasipolynomial
in $\alpha p^{-r}$, that is,
$$c(r,\alpha)=O(\exp(\log^{O(1)}(\alpha^{-1}p^r))).$$ 
\end{prop}
Since the statement of Proposition \ref{codimensionReduction} looks complicated, an explanation is in order. We can think of the maps $\xi_1,\ldots,\xi_s$ as the number of simultaneous constraints that the spaces $(V_y)_{y \in A}$ have to satisfy. (Thus at the beginning, $s=0$ since we don't have any information on the $V_y$ yet.) At each step, either the codimension of the $V_y$ in $V$ is reduced (the second alternative) or the number of constraints is increased (the third alternative). Clearly this process must stop when $r=0$ (i.e., $V_y=V$ for all $y$) or when $s=r$ (i.e., all the $V_y$ are given by $r$ simultaneous constraints). In either case the $V_y$ are very structured, which gives us the desired bilinear structure. We will now make this argument rigorous while keeping track of the bounds.
 
 \begin{proof}[Proof of Theorem \ref{mainTrm} using Proposition \ref{codimensionReduction}]
Applying Corollary \ref{cor:regularity} with $\epsilon=p^{-r}/256$ and $r=O(\log^{O(1)}\alpha^{-1})$, 
we obtain an affine subspace $W^{(0)}$ of $V$ of codimension 
\[
O((\epsilon \alpha)^{-1/2} \log \alpha^{-1}) = O(p^{r/2} \alpha^{-2/3})
\]
such that the set $A_0 := A \cap W^{(0)}$ has density $\alpha_0 \geq \alpha$ and is $\epsilon$-pseudorandom in $W_0$. 
We set $V^{(0)}=V, P_0 = \cup_{y\in A_0} V_y\times\{y\} \subset P$ and apply Proposition \ref{codimensionReduction} with the tuple $(V^{(0)}, W^{(0)}, A_0, P_0)$ and $s_0=0,r_0=r$. 

If the first alternative of Proposition \ref{codimensionReduction} holds, we stop.

Suppose the second alternative of Proposition \ref{codimensionReduction} holds. We set $V^{(1)} \subset V^{(0)}$ to be the subspace $V'$ given by the second alternative, of codimension $O(r)$. We are also given subspaces $V^{(1)}_y\leq V^{(1)}$ of codimension at most $r_1= r-1$
such that
$P\supset \bigcup_{y\in A}V^{(1)}_y\times\{y\}$.

Suppose the third alternative holds.
We obtain a set $S \subset A_0$ of density $c(r, \alpha_0)$ in $W^{(0)}$, an affine map $\xi_1 : W_0\rightarrow V^*$
and subspaces
$U^{(1)}_y\leq V_y^*$ of dimension at most $r_1\leq
r-1$ such that 
$V_y^\perp=\Span(\xi_1(y))+U^{(1)}_y$.
Then we let $V^{(1)}=V$ and $V^{(1)}_y=V_y$.
We can find an affine subspace $W^{(1)} \subset W^{(0)}$ of codimension $O(p^{r/2} \alpha_0^{-2/3})$ in $W^{(0)}$
such that the set $A_1 := S \cap W^{(1)}$ is $\epsilon$-pseudorandom and has density $\alpha_1 \geq c(r, \alpha_0)$ in $W^{(1)}$. Let $s_1=1$ and $r_1=r$.

Set $P_1 = \bigcup_{y\in A_1} V^{(1)}_y\times\{y\}$. We have $P_0\supset P_1$.

We can now apply Proposition \ref{codimensionReduction} with $(V^{(1)}, W^{(1)}, A_1, P_1,r_1,s_1)$ and start an iterative process. 
This iterative process stops whenever one can
apply the first item of Proposition \ref{codimensionReduction}, or when
$r-s$ vanishes. When applying either of the last
two alternatives, at least one of the parameters $r$ or $r-s$
is decreased by at least one, while the other one cannot increase, so the iteration does eventually stop.

At the $i$-th stage, we obtain a subspace $V^{(i)} \subset V$ of codimension $O(ri)$, an affine subspace $W^{(i)} \subset W$ of codimension
\[
O(\exp(\log^{C^i}\alpha^{-1})),
\]
where $C$ is a constant (depending at most on $p$),
an $\epsilon$-uniform set $A_i \subset W^{(i)}$ of density  
\[
\alpha_i = \Omega(\exp(-\log^{C^i}\alpha^{-1}))
\]
and a set
\[
P_i = \cup_{y\in A_i} V^{(i)}_y\times\{y\} \subset V^{(i)} \times W^{(i)}
\]
where each $V^{(i)}_y \subset V^{(i)}$ has codimension $r_i\leq r$. 
Besides, we have affine maps
$\xi_1,\ldots,\xi_{s_i}$ from $W^{i}$ to $V^{(i)*}$
and subspaces $U^{(i)}_y\leq V^{(i)*}$ of dimension at most
$r_i-s_i$
such that $(V^{(i)}_y)^\perp=\Span(\xi_1(y),\cdots,\xi_{s_1}(y))+U^{(i)}_y$. Furthermore, 
$
P \supset P_i.
$
Suppose the algorithm stops after the $i$-th iteration, where $i \leq 2r$. Note that we have $s_i\leq r$,
\[
\codim V^{(i)} = O(r^2) = O(\log^{O(1)}\alpha^{-1}),
\]
and 
\[
\codim W^{(i)}=O(\exp(\log^{C^r}\alpha^{-1}))=O(\exp(\exp(\exp(\log^{O(1)}\alpha^{-1})))).
\]
There are two possibilities.

\noindent \textbf{Case 1:} 
$r_i=s_i$, and
\[
P_i= \{ (x,y) \in V^{(i)}\times A_i \mid  x\cdot\xi_1(y)=\cdots=x\cdot\xi_{s_i}(y)=0 \} 
\]
where $ \xi_1, \ldots, \xi_{s_i}$ are affine maps from $W^{(i)}$ to $V^{(i)*}$, $A_i \subset W^{(i)}$ is a set of density $\gamma:=\alpha_i = \Omega \left( \exp(-\exp(\exp(\log^{O(1)}\alpha^{-1}))) \right)$.

\noindent \textbf{Case 2:} The first alternative of Proposition \ref{codimensionReduction} holds, and 
\[ \phi_V(P_i) \supset \{(x,y)\in X \times W_2 \mid x\cdot \overrightarrow{\xi_1}(y)=\cdots=
x\cdot \overrightarrow{\xi_{s_i}}(y)=0\},
\]
where $ \xi_1, \ldots, \xi_{s_i}$ are affine maps from $W^{(i)}$ to $V^{(i)*}$, $X \subset V^{(i)}$ is a set of density 
$\Omega \left( p^{-r_i} \right) = \Omega \left( \exp \left( - \log^{O(1)} \alpha^{-1} \right)\right)$ and $W_2$ is the direction of $W^{(i)}$.

Since the two cases are similar, we will work with Case 1. By translating $P$ by $(0,a)$ for some $a \in A_i$ if necessary, we may assume that $W^{(i)}$ is a vector subspace of $V$. 
Let $\eta:=  \frac{1}{10} \gamma^{3/2} p^{-r-1} $. Applying Lemma \ref{ExtendedRegularity} with $t=r+1$, there is a subspace $H\leq W^{(i)}$ of codimension 
$O(r \eta^{-1} \log\gamma^{-1})$ such 
that $\abs{A'}= \gamma'\abs{H}$ (where $A'=A_{i}\cap H$) with $\gamma'\geq \gamma$
and for any subspace $F$ of codimension at most $r+1$ of $H$, $\frac{\abs{A'\cap F}}{\abs{F}}\leq \gamma(1+\eta)$.

For each $x \in V^{(i)}$, let $B_x = \{ y \in H \mid x\cdot \xi_1(y)=\cdots=
x\cdot \xi_{s_i}(y)=0 \}$. Then $B_x$ is a subspace of codimension at most $r$ inside $H$.
Let $A_x = A'\cap B_x$. We claim that $2A_x - 2A_x = \overrightarrow{B_x}$. 

By Lemma \ref{lem:wholespace}, it suffices to show
$\abs{\widehat{1_{A_x}}(\chi)}<\gamma_x^{3/2}$ for any $\chi\neq 1$, where $\gamma_x$ is the density of
$A_x$ in $B_x$. 

Suppose for a contradiction that this is not true. Then Lemma \ref{lem:densityincrement} implies that there  is a hyperplane $F$ of $B_x$ 
on which the density of $A$ is at least $\gamma_x+\gamma_x^{3/2}/2$. 
From Lemma \ref{ExtendedRegularity} we also have $\gamma_x\geq \gamma(1-\eta p^{r+1})$.
Therefore, 
\begin{eqnarray}
\gamma_x+\gamma_x^{3/2} /2 &\geq & \gamma(1-\eta p^{r+1})+\frac{1}{2} \gamma^{3/2}(1-\eta p^{r+1})^{3/2} \nonumber \\
&\geq& \gamma(1-\eta p^{r+1}) + \frac{1}{2} \gamma^{3/2} (1 - 2 \eta p^{r+1}) \nonumber \\
&\geq& \gamma - \frac{1}{10} \gamma^{3/2} + \frac{2}{5} \gamma^{3/2} = \gamma + \frac{3}{10} \gamma^{3/2} > \gamma + \eta. 
\end{eqnarray}
This contradicts the assumption on $H$ since $F$ is a subspace of codimension at most $r+1$ of $H$. Therefore, $2A_x - 2A_x = \overrightarrow{B_x}$ and
\[
\phi_V(P) \supset \cup_{x \in V^{(i)}} \{x\}\times \overrightarrow{B_x} 
= \{(x,y) \in V^{(i)} \times H \mid x\cdot \overrightarrow{\xi_1}(y)=\cdots=
x\cdot \overrightarrow{\xi}_{s_i}(y)=0 \}.
\]
Since the codimension of $H$ in $W^{(i)}$ is $O(r  \eta^{-1} \log\gamma^{-1})$, Theorem \ref{mainTrm} follows in this case. In Case 2, a similar argument shows that $\phi_H \phi_V (P)$ contains the desired bilinear structure.
\end{proof}

\section{Proof of Proposition \ref{codimensionReduction}} \label{sec:proof}

First we suppose that there exists a nonzero $\boldsymbol{\lambda} = (\lambda_1, \ldots, \lambda_s) \in \F_p^s$
such that $\xi'_0=\sum_{j=1}^s \lambda_j\xi_j$
satisfies $\rk \overrightarrow{\xi'_0}<3r+10$.
By completing $\lambda$ into a basis of
$\F_p^s$, we obtain affine maps $\xi'_0,\ldots,\xi'_{s-1}$
from $W$ to $V^*$ such that
$\Span (\xi'_0(y),\ldots,\xi'_{s-1}(y))=\Span (\xi_1(y),\ldots,\xi_s(y))$ for any $y\in A$.
Let $V'=\Span(\xi'_0(y)\mid y\in A)^\perp$, then the codimension of $V'$ in $V$ is $\leq 3r +10$.
For each $y \in A$, let $V'_y < V'$ be given by $(V'_y)^\perp=\Span(\xi'_1(y),\ldots,\xi'_{s-1}(y))+U'_y$
where $U'_y$ is the projection of $U_y$ onto $V'^*$. Then $\dim U'_y\leq r-s$ and $\codim_{V'} V'_y \leq r-1$. Since $V'_y < V_y$, we have $P\supset \bigcup_{y\in A}V'_y\times \{y\}$. This proves the second alternative (with $S=A$).

So let us now suppose that there exists no nonzero $\lambda \in \F_p^s$ such that $\xi'_0=\sum_{j\in [s]}\lambda_j\xi_j$
satisfies $\rk \overrightarrow{\xi'_0}<3r+10$. For $x \in V$, let $A_x=\{y\in A\mid x\in V_y\} \subset W$. Thus 
\begin{equation} \label{eq:ax}
\sum_{x \in V} |A_x| = |P| = |A| |V| p^{-r} = \alpha p^{-r} |W||V|.  
\end{equation}
Also, let $B_x=\{y\in W\mid x\cdot \xi_1(y)=\cdots =x\cdot \xi_s(y)=0\}$; it is an affine subspace of codimension at most $s$, and $A_x\subset A\cap B_x$. 

\noindent \textbf{Claim 1:} $\P_{x\in V}(\codim B_x<s)\leq \epsilon p^{-r}/4$. (Recall that $\epsilon = p^{-r}/256$.)

Note that 
$$\overrightarrow{B_x}=\{y\in \overrightarrow{W}\mid
y\cdot \overrightarrow{\xi_1}^T(x)=\cdots=y\cdot \overrightarrow{\xi_s}^T(x)=0
\}$$
is a subspace of codimension $s$ unless there exists
a nonzero $\lambda\in \F_p^s$ such that
$\sum_{j\in [s]}\lambda_j\overrightarrow{\xi_j}^T(x)=0$.
For any fixed such $\lambda$, the set of $x$
that satisfy this relation
is a linear subspace, namely the kernel $K_\lambda$ of
$\sum_{j\in [s]}\lambda_j\overrightarrow{\xi_j}^T$
whose codimension equals the rank of
$\sum_{j\in [s]}\lambda_j\overrightarrow{\xi_j}$,
hence at least $3r+10$.
Hence $\abs{K_\lambda}/\abs{V}\leq p^{-3r-10}$.
Because there are at most $p^r$ tuples $\lambda$ to consider,
we conclude that $\P_{x\in V}(\codim B_x<s)\leq p^{r} \cdot p^{-3r-10} \leq \epsilon p^{-r}/4$, and Claim 1 is proved.


\noindent \textbf{Claim 2:} Let $\alpha_x$ be the density of $A_x$ in $B_x$, then $\E_{x\in V}\alpha_x \geq \alpha p^{s-r}(1-\epsilon/4)$.

Indeed, let $X=\{x\in V\mid\codim B_x=s\}$, then we have
\begin{align*}
\E_{x\in V}\alpha_x &=
\frac{1}{\abs{V}}\sum_{x\in V}\frac{\abs{A_x}}{\abs{B_x}} \geq \frac{1}{\abs{V}}\sum_{x\in X}\frac{\abs{A_x}}{\abs{B_x}}\\
&= p^s\E_{x\in V}\frac{\abs{A_x}}{\abs{W}}-p^s\frac{1}{\abs{V}}\sum_{x\in X^c}\frac{\abs{A_x}}{\abs{W}}\\
&\geq 
\alpha p^{s-r}-\alpha p^{s}\frac{\abs{X^c}}{\abs{V}} \geq \alpha p^{s-r}(1-\epsilon/4)
\end{align*}
where we have used \eqref{eq:ax} and the trivial bound $|A_x| \leq |A| = \alpha |W|$. 

Proposition \ref{codimensionReduction} will follow from Lemmas \ref{randomnessLinearity} and \ref{linearity}.

\begin{lm}
\label{randomnessLinearity}
At least one of the following statements holds.
\begin{enumerate}
\item For at least $p^{-r}\abs{V}/12$ elements $x\in V$,
we have $2A_x-2A_x=\overrightarrow{B_x}$ (the direction of $B_x$).
\item Among additive quadruples $y_1+y_2=y_3+y_4$ in $A$, a proportion at least $p^{-4r}\epsilon$ has the property that $\codim\bigcap_{i=1}^4V_{y_i}<4r-s$.

\end{enumerate}
\end{lm}
\begin{proof}
Let $Q$ be the set of additive quadruples $\mathbf{y}=(y_1,\ldots,y_4)$ of $A$. Let $m=\dim W$.
We have
\begin{align*}
\E_{x\in V}\nor{1_{A_x}}_{U^2(B_x)}^4 
&=\E_{x\in V}
\E_{\substack{y_1,\ldots,y_4\in B_x\\y_1+y_2=y_3+y_4}}\prod_{i=1}^4 1_{y_i\in A_x}
\\
&\leq\frac{1}{p^{3(m-s)}}\sum_{\substack{(y_1,\ldots,y_4)\in A^4 \\
y_1+y_2=y_3+y_4}}\E_{x\in V}1_{\forall i\, x\in V_{y_i}} \qquad \textrm{(since $\dim B_x \geq m-s$)}\\
&=\frac{1}{p^{3(m-s)}}\sum_{\mathbf{y}\in Q}
p^{-\codim \bigcap_i V_{y_i}}\\
&\leq \alpha^4(1+\epsilon)\big(\E_{\mathbf{y}\in Q}(p^{3s-\codim \bigcap_i V_{y_i}})\big) \qquad \textrm{(recall that $A$ is $\epsilon$-uniform)}\\
&\leq\alpha^4p^{-4(r-s)}(1+\epsilon)(1+p^{4r-s}\P_{\mathbf{y}\in Q}(\codim \bigcap V_{y_i}<4r-s).
\end{align*}

So either
\begin{equation}
\label{additiveStruc}
\P_{\mathbf{y}\in Q}\left(\codim\bigcap_{i=1}^4V_{y_i}<4r-s\right)\geq p^{-4r+s}\epsilon
\end{equation}
or
\begin{equation}\label{psr}
\E_{x\in V}\nor{1_{A_x}}_{U^2(B_x)}^4\leq \alpha^4p^{-4(r-s)}(1+\epsilon)^2.
\end{equation}
Equation \eqref{additiveStruc}
is exactly the second clause of Lemma \ref{randomnessLinearity},
so assume instead that
\eqref{psr} holds.
We infer that
\begin{equation*}
\begin{split}
\E_{x\in V}\nor{1_{A_x}-\alpha_x}_{U^2}^4
&=
\E_{x\in V}(\nor{1_{A_x}}_{U^2}^4-\alpha_x^4)
\\
&\leq
\E_{x\in V}\nor{1_{A_x}}_{U^2}^4-\alpha^4p^{-4(r-s)}(1-\epsilon/4)^4\\
&\leq \alpha^4p^{-4(r-s)}(2\epsilon  +\epsilon^2+\epsilon)\\
&\leq 4\epsilon\alpha^4p^{-4(r-s)}=:\gamma
\end{split}
\end{equation*}
where we used Jensen's inequality, the lower bound
$\E_{x\in V}\alpha_x\geq p^{-(r-s)}(1-\epsilon p^{-r}/4)$
and the elementary inequality 
$(1-\epsilon/4)^4\geq 1-\epsilon$.
Thus, if $X_1:=\{x\in V\mid \nor{1_{A_x}-\alpha_x}_{U^2(B_x)}^4\leq 4p^r\gamma \}$,
by Markov's inequality,
we have $\abs{X_1}\geq \abs{V}(1-p^{-r}/4)$.
Also, because $\alpha_x\leq \alpha p^s$ for any $x\in V$,
the set $X_2:=\{x\in V\mid\alpha_x > \alpha p^{-(r-s)}/2\}$
has density at least $p^{-r}(1/2-\epsilon/4)\geq p^{-r}/3$. So $X_3:=X_1\cap X_2$
must have
density at least $p^{-r}/12$ by inclusion-exclusion.

Besides, if $\epsilon=1/(256p^r)$, then for $x\in X_3$ we have
$\nor{1_{A_x}-\alpha_x}_{U^2}^4 \leq 4 p^r \gamma < \alpha_x^4$ and then $2A_x-2A_x=\overrightarrow{B_x}$ by Lemma \ref{lem:wholespace}.
\end{proof}

We now prove Proposition \ref{codimensionReduction}.
When the first outcome of Lemma \ref{randomnessLinearity}
holds, we see that
$\phi_V(P)$
contains $\{(x,y)\in X_3\times W_2\mid x\cdot \overrightarrow{\xi_1}(y)=\cdots=
x\cdot \overrightarrow{\xi_s}(y)=0\}$ where $W_2$ is the direction of $W$.

The real challenge lies
in extracting something from the second outcome of Lemma \ref{randomnessLinearity}.
This is the purpose of the next lemma.
\begin{lm}
\label{linearity}
Suppose $r>s$ and a proportion
at least $\kappa$ of the additive quadruples $(y_1,\ldots,y_4)$ of $A$ 
have the property that $\codim \bigcap_{i=1}^4V_{y_i}<4r-s$.
Then there is a subset $S\subset A$ of density $\sigma=\sigma(r,\alpha,\kappa)$ such that one of the following holds.
\begin{enumerate}
\item There is a subspace $V'\leq V$ of codimension one such that $V_y \subset V'$ for all $y \in S$. Besides, there exist affine maps $\xi'_1,\ldots,\xi'_{s-1}$
from $W$ to $V'^*$ and spaces $U'_y\leq V'^*$ of dimension
at most $r-s$
such that $(V_y)^{\perp}=\Span(\xi'_1(y),\ldots,\xi'_{s-1}(y))+U'_y$ (where $(V_y)^{\perp}$ is defined inside $V'^*$).
\item
There is an affine map $\xi_{s+1} : W\rightarrow V^*$ and
a subspace $U'_y\leq V^*$ of dimension at most $r-s-1$ such
that
$V_y^\perp=\Span(\xi_j(y)\mid j\in [s+1])+U'_y$ for all $y \in S$.
\end{enumerate} 
Moreover $\sigma$ can be taken to be quasipolynomial\footnote{Polynomial under the polynomial Freiman-Ruzsa conjecture.} in $\alpha 
\kappa p^{-r}$.
\end{lm}
Applying Lemma \ref{linearity} with $\kappa = p^{-4r} \epsilon = p^{-5r}/256$, the first alternative implies the second statement of Proposition \ref{codimensionReduction} since $\codim_{V'} V_y \leq r-1$, while the second alternative yields the third one of Proposition \ref{codimensionReduction}. Our goal is now to prove Lemma \ref{linearity}.



\begin{proof}[Proof of Lemma \ref{linearity}]
Let $\xi_{s+1},\ldots,\xi_r$ be (not necessarily linear) maps from $A$ to
$V^*$ such that $U_y=\Span(\xi_{s+1}(y),\ldots,\xi_r(y))$
for any $y\in A$.
The number of additive quadruples in $A$ is at least $\alpha^4\abs{W}^3=\alpha\abs{A}^3$, and we assume at least $\kappa \alpha\abs{A}^3$ of them have the property
that the $4r$ vectors $\xi_j(y_i)$ satisfy at least
$s+1$ linearly
independent equations.
For any additive quadruple in $A$, we already have
$s$ obvious equations 
\begin{equation}
\label{obviousEqs}
\xi_j(y_1)+\xi_j(y_2)=\xi_j(y_3)+\xi_j(y_4)\quad\text{ for }
j\in [s],
\end{equation}
so there needs to be one more (independent) equation.
Because there are only $p^{4r}$ possible linear equations
\begin{equation}\label{liaison} \sum_{j=1}^r\sum_{i=1}^4a_{i,j}\xi_j(y_i)=0,
\end{equation} 
the pigeonhole principle implies that we can find
$(a_{i,j})\in \F^{4r}\setminus \{0\}$ (linearly independent from the vectors $b_{i,j}=1_{j=j_0}$
for $j_0\in [s]$)
such that there are at least
$\kappa\alpha\abs{A}^3/p^{4r}$ quadruples $(y_1,\ldots,y_4)\in W^4$
for which $y_1+y_2=y_3+y_4$
and
equation \eqref{liaison} holds.
Let $T$ be that set of quadruples.
Write $\mathbf{a}_i=(a_{i,j})_{j=1,\ldots,r}$. We distinguish two cases.

\noindent \textbf{Case 1:} One of the four families $\mathbf{a}_1,\ldots,\mathbf{a}_4$,  say $\mathbf{a}_4$, satisfies $a_{4,j}=0$
for any $j>s$. Then we can use the equations \eqref{obviousEqs}
to eliminate $y_4$ in equation \eqref{liaison}.
We obtain $\phi_1+\phi_2+\phi_3=0$ for some vectors $\phi_i\in V_{y_i}^\perp$ for $i\in [3]$, not all equal to 0.
Write $r(\phi)=\abs{\{y\in A\mid \phi\in V_y^\perp\}}$ for any $\phi\in V^*$.
Then we have
$$
p^{-4r}\kappa \alpha\abs{A}^3
\leq \abs{T}\leq
\sum_{\substack{\phi_1,\phi_2\in V^*,\\ (\phi_1,\phi_2) \neq (0,0)}}r(\phi_1)r(\phi_2)r(-\phi_1-\phi_2)\leq 2 \max_{\phi\in V^* \setminus \{0\}} r(\phi)\left(\sum_{\phi\in V}r(\phi)\right)^2.
$$
A double counting argument shows that
$\sum_{\phi\in V^*}r(\phi)\leq p^r\abs{A}$.
So $\max_{\phi\in V^* \setminus \{0\}} r(\phi)\geq \frac{1}{2} \kappa\alpha p^{-6r}\abs{A}$, which implies that there exists a linear form $\phi\in V^* \setminus \{0\}$ such that for a positive
proportion of $y\in A$, we have $\phi\in V_y^\perp$. Name $S_0 \subset A$ this set of elements $y\in A$, then $|S_0| \geq \frac{1}{2} \kappa\alpha p^{-6r}\abs{A}$. Let $V'=\phi^\perp$, then $\codim_{V}V' =1$ and $V_y \leq V'$ for any $y \in S_0$. Let $S_1 := \{y \in A: \phi \in \Span(\xi_{1}(y),\ldots,\xi_s(y)) \}$ and $S_2 := S_0 \setminus S_1$.
So 
$\max (\abs{S_0},\abs{S_1})\geq \frac{1}{4} \kappa\alpha p^{-6r}\abs{A}$.
We distinguish two subcases.

\noindent \textit{Case 1a:} $|S_1| \geq \frac{1}{4} \kappa\alpha p^{-6r}\abs{A}$. Since there are at most $p^r$ possible ways to write $\phi$ as a linear combination of $\xi_{1}(y),\ldots,\xi_s(y)$, there is a nonzero $\boldsymbol{\lambda} = (\lambda_1, \ldots, \lambda_s) \in \F_p^s$ and a subset $S_3 \subset S_1$ of size $\gg \kappa \alpha p^{-7r}$ such that $\phi = \sum_{i=1}^s \lambda_i \xi_i(y)$ for any $y \in S_3$. 
By completing $\boldsymbol{\lambda}$ into a basis of $\F_p^s$, we can find affine maps 
$\xi'_1, \ldots, \xi'_{s-1} : W \rightarrow V^*$ such that $\Span(\xi_{1}(y),\ldots,\xi_s(y)) = \Span(\phi, \xi'_{1}(y),\ldots,\xi'_{s-1}(y))$ for any $y \in S_3$. 
Considering now $V_y$ as a subspace of $V'$, and thus defining its orthogonal as a subspace of $V'^*$,
we have
\[
V_y^\perp = \Span(\xi'_{1}(y),\ldots,\xi'_{s-1}(y)) + U'_y
\]
for any $y \in S_3$, where $U'_y$ is the projection of $U_y$ on $V'^*$. Thus the first alternative of Lemma \ref{linearity} follows with $S=S_3$.

\noindent \textit{Case 1b:} $|S_2| \geq \frac{1}{4} \kappa\alpha p^{-6r}\abs{A}$. For $y \in S_2$, let $U''_y \leq V_y^\perp$ be such that 
$$V_y^\perp = \Span(\xi_{1}(y),\ldots,\xi_{s}(y), \phi) \oplus U''_y,$$
 then $\dim U''_y \leq r-s-1$. Projecting to $V'$, we have
\[
V_y^\perp = \Span(\xi_{1}(y),\ldots,\xi_{s}(y)) + U''_y
\]
for any $y \in S_2$ and letting $U'_y = U''_y+\Span(\xi_s(y))$, the first alternative of Lemma \ref{linearity} follows with $S=S_2$.


\noindent \textbf{Case 2:} None of the four families satisfies
$a_{i,j}=0$ for any $j>s$. We shall aim at linearity instead of constancy.
Removing quadruples for which two entries are equal (there are
$O(\abs{A}^2)$ such quadruples), we still have a set $T'$ of quadruples satisfying $\abs{T'}\geq C\abs{A}^3$ for some constant
$C\gg \kappa\alpha p^{-4r}$. Applying Lemma \ref{lem:partition}, we can pick a partition
$A=A_1\cup A_2\cup A_3\cup A_4$ such that the set
$$T''=T'\cap A_1\times A_2\times A_3\times A_4$$
satisfies
$\abs{T''}\geq C\abs{A}^{3}/256$.
For $i\in[4]$ and $y\in A_i$, set $\xi'_{s+1}(y)=
z_i\sum_{j\in[r]} a_{i,j}\xi_j(y)$ 
where $z_1=z_2=1$ and $z_3=z_4=-1$.
Observe that 
$\xi'_{s+1}(y)$ is
a nonzero vector in $V_y^\perp$.
For $i\in [4]$, let $j_i>s$ be any index such that
$a_{i,j_i}\neq 0$. For $y\in A_i$, let
$U'_y=\Span(\xi_{s+1}(y),\cdots,\widehat{\xi_{j_i}(y)},\cdots,
\xi_r(y))$,
where the hat denotes an omitted form; this is a space of dimension at most $r-s-1$.
We have $V_y^\perp=\Span(\xi_{1}(y),\ldots,\xi_s(y),\xi'_{s+1}(y))+U'_y$.
Further, for a quadruple $\mathbf{y}\in T''$, we observe that
$(y_i,\xi_{s+1}'(y_i))_{i\in[4]}$ is an additive quadruple. So there are at least $C\abs{A}^{3}/256$ additive quadruples in the graph  
$\{(y,\xi'_{s+1}(y))\mid y\in A\}$. 
We then invoke Lemma 9 to obtain a set $S\subset A$
of quasipolynomial (in $C$) density in $A$,
such that $\xi'_{s+1}$
coincides with an affine map on $S$.
This concludes the proof.
\end{proof}

\subsection*{Acknowledgements}
The authors are thankful to Ben Green and Terence Tao 
for suggesting the bilinear Bogolyubov theorem
and sharing ideas on how it could be proven. The first author
is also grateful to his supervisor Julia Wolf and Tom Sanders
for useful conversations. 
Part of this work was carried out while the first author was staying at the Simons Institute for the theory of computing and supported by a travel grant of the University of Bristol Alumni Foundation.
The second author was supported by National Science Foundation Grant DMS-1702296 and a Ralph E. Powe Junior Faculty Enhancement Award from Oak Ridge Associated Universities.

\end{document}